\documentclass [11pt,a4paper]{article}
\usepackage{amsfonts,amssymb,amsmath,amscd,latexsym,makeidx,theorem}
\usepackage[latin1]{inputenc}

\author{Andrea Malchiodi \\ \footnotesize SISSA, Trieste \\ \footnotesize \texttt{malchiod@sissa.it} \and  Luca Martinazzi \\ \footnotesize Rutgers University \\ \footnotesize \texttt{martinazzi@math.rutgers.edu}}
\title{Critical points of the Moser-Trudinger functional on a disk}
\newtheorem{trm}{Theorem}

\newtheorem{lemma}[trm]{Lemma}

\newcommand{\R}[1]{\mathbb{R}^{#1}}

\newcommand{\de}{\partial}
\newcommand{\ve}{\varepsilon}
\newcommand{\M}[1]{\mathcal{#1}}

\newenvironment{proof}{\noindent\emph{Proof.}}{\phantom{ } \hfill$\square$\medskip}

\DeclareMathOperator{\loc}{loc}

\renewcommand{\l }{\lambda }
\renewcommand{\L }{\Lambda }

\newcommand{\be}{\begin{equation*}}
\newcommand{\ee}{\end{equation*}}

\newcommand{\e }{\varepsilon }

\newcommand{\bdm}{\begin{displaymath}}
\newcommand{\edm}{\end{displaymath}}

\newcommand{\n }{\nabla }

\newcommand{\D }{\Delta }

\begin{document}
\maketitle

\begin{abstract} \noindent On the unit disk $B_1\subset \R{2}$ we study the Moser-Trudinger functional
$$E(u)=\int_{B_1}\Big(e^{u^2}-1\Big)dx,\quad u\in H^1_0(B_1)$$
and its restrictions $E|_{M_\Lambda}$, where $M_{\Lambda}:=\{u\in H^1_0(B_1):\|u\|^2_{H^1_0}=\Lambda\}$ for $\Lambda>0$. We prove that if a sequence $u_k$ of positive critical points of $E|_{M_{\Lambda_k}}$ (for some $\Lambda_k>0$)
blows up as $k\to\infty$, then $\Lambda_k\to 4\pi$, and $u_k\to 0$ weakly in $H^1_0(B_1)$ and strongly in $C^1_{\loc}(\overline B_1\setminus\{0\})$.

Using this we also prove that when $\Lambda$ is large enough, then $E|_{M_\Lambda}$ has no positive critical point, complementing previous existence results by Carleson-Chang, M. Struwe and Lamm-Robert-Struwe.
\end{abstract}

\begin{center}

\bigskip\bigskip

\noindent{\it Key Words:} Moser-Trudinger inequality, critical points, blow-up analysis,
variational methods

\bigskip

\centerline{\bf AMS subject classification: 35B33,  35B44, 35A01, 34E05}

\end{center}

\section{Introduction}

Let $\Omega\Subset\R{2}$ be a smooth bounded and connected open set. It is well-known that there is a Sobolev embedding $W^{1,p}_0(\Omega) \hookrightarrow L^\frac{2p}{2-p}(\Omega)$ for $p\in [1,2)$, but $H^1_0(\Omega):=W^{1,2}_0(\Omega) \not \hookrightarrow L^\infty(\Omega)$. However it was proven by N. Trudinger \cite{tru} that $e^{u^2}\in L^1(\Omega)$ whenever $u\in H^1_0(\Omega)$. This embedding was sharpened by J. Moser \cite{mos} who showed that
\begin{equation}\label{eq:MT}
    \sup_{u\in H^1_0(\Omega),\; \|u\|_{H^1_0}^2\le 4\pi} \; \int_\Omega\Big( e^{u^2}-1\Big) dx \leq C|\Omega|, \qquad \|u\|_{H^1_0}:=\bigg(\int_\Omega |\n u|^2 dx \bigg)^{\frac{1}{2}},
\end{equation}
and
\begin{equation}\label{eq:MT2}
    \sup_{u\in H^1_0(\Omega),\; \|u\|_{H^1_0}^2\le 4\pi+\delta}\; \int_\Omega\Big( e^{u^2}-1\Big)
    dx =+\infty, \qquad \text{for every }\delta>0.
\end{equation}
Since then, a formidable amount of work has been devoted to the study of the functional
$$E(u):=\int_\Omega\Big( e^{u^2}-1\Big)dx,\qquad u\in H^1_0(\Omega)$$
and in particular of its critical points. Clearly $u\equiv 0$ is the only global minimum of $E$, but because of \eqref{eq:MT2} we cannot look for a global maximizer of $E$ in $H^1_0$. Instead one might hope to find a maximizer of $E|_{M_\Lambda}$, i.e. of $E$ constrained to the manifold
$$M_\Lambda:=\Big\{u\in H^1_0(\Omega):\|u\|^2_{H^1_0}=\Lambda\Big\}$$
for $\Lambda\in (0,4\pi]$, or to find other kinds of critical points (local maxima or minima, saddle points, etc.) when $\Lambda>4\pi$. As long as $\Lambda<4\pi$ the embedding \eqref{eq:MT} is in fact compact, so the existence of a maximizer is elementary, but when $\Lambda\ge 4\pi$  compactness is lost and also the Palais-Smale condition does not hold anymore, see \cite{AP}.

In spite of these difficulties Carleson and Chang \cite{CC} proved that when $\Omega=B_1(0)$ (the unit disk in $\R{2}$)  $E|_{M_{4\pi}}$ has a maximizer. This result was extended by Struwe \cite{sIHP} who proved the existence of a maximizer in $M_{4\pi}$ when $\Omega$ is close to a ball, and finally by Flucher \cite{flu} for any bounded smooth $\Omega$ (see also
\cite{DDR} for a related result in higher dimension).

The existence of critical points on $M_\Lambda$ in the supercritical regime, i.e. for $\Lambda>4\pi$, is even more challenging, and to the fundamental question of the existence of critical points of $E|_{M_{\Lambda}}$ for $\Lambda$ large only few answers have been given.  Monahan \cite{mon} gave numerical evidence that when $\Omega=B_1(0)$ then for some $\Lambda^*>4\pi$ the functional $E|_{M_\Lambda}$ has a local maximum and a mountain pass critical point for every $\Lambda\in (4\pi,\Lambda^*)$. Assuming that a local maximum of $E|_{M_{4\pi}}$ exists (which was later shown to be true for arbitrary domains by Flucher \cite{flu}) Struwe proved in \cite{sIHP} that for some $\L^* = \L^*(\Omega) > 4 \pi$ and
for a.e. $\L \in (4 \pi, \L^*)$ two critical points exists. This result was then extended in \cite{LRS} to
all values of $\L \in (4 \pi, \L^*)$ through the more precise information given by a parabolic flow,
compared to the one given by the Palais-Smale condition.

Further, using implicit function methods, Del Pino, Musso and Ruf \cite{DMR2} were able to characterize some
of these critical points as one-peaked bubbling functions which blow-up as $\L \searrow 4 \pi$. In the same
paper they showed that if $\Omega$ is not contractible, then for some $\Lambda^\dag>8\pi$ the functional $E|_{M_\Lambda}$ has a critical point of multi-peak type for $\Lambda\in (8\pi, \Lambda^\dag)$. When $\Omega$ is a radially symmetric annulus they also proved for any $1\le \ell\in\mathbb{N}$ the existence of some $\Lambda_\ell^* >4\pi\ell$ such that $E|_{M_\Lambda}$ has a critical point when $\Lambda\in (4\pi\ell, \Lambda_\ell^*)$.
We also refer to \cite{sJEMS} and \cite{LRS}
for related results on domains with {\em small holes}, in the spirit
of \cite{cor} (where the Yamabe equation was treated).

\

The previous results, in particular those in \cite{DMR2}, suggest that at least when $\Omega$ is not contractible $E|_{M_\Lambda}$ might have critical points even when $\Lambda$ is much larger that $4\pi$.
In this paper we will show that such a topological assumption on $\Omega$ is natural. In fact we will prove that when $\Omega= B_1(0)$, then $E|_{M_\Lambda}$ has no positive critical points for $\Lambda$ large enough.

\begin{trm}\label{nonex} For $\Omega= B_1(0)$ there exists $\Lambda^{\sharp}> 4\pi$ such that the functional $E|_{M_\Lambda}$ has
\begin{itemize}
\item[(i)] no positive critical points for $\Lambda>\Lambda^{\sharp}$,
\item[(ii)]  at least $2$ positive critical points for $\Lambda\in (4\pi, \Lambda^{\sharp})$,
\item[(iii)] at least one positive critical point for $\Lambda\in (0,4\pi]\cup \{\Lambda^{\sharp}\}$.
\end{itemize}
\end{trm}

The proof of the non-existence part in Theorem \ref{nonex} (Part (i)) will be completely self-contained.  To prove Parts (ii) and (iii) we will also use Theorem 1.7 from \cite{sIHP}, which gives the existence of some $\Lambda^*>4\pi$  such that a positive critical point of $E|_{M_\Lambda}$ exists for $\Lambda\in (4\pi,\Lambda^*)$. Actually by  \cite[Theorem 1.8]{sIHP} and \cite[Theorem 6.5]{LRS}, $E|_{M_\Lambda}$ has two positive critical points whenever $\Lambda\in (4\pi,\Lambda_*)$, for some $\Lambda_*\in (4\pi,\Lambda^*]$. Then Theorem \ref{nonex} complements  these results by showing that, at least when $\Omega=B_1$, the existence of two positive critical points of $E|_{M_\Lambda}$  for $\Lambda>4\pi$ persists until we reach the energy threshold $\Lambda=\Lambda^\sharp$, beyond which we have non-existence. A qualitatively similar picture
has also been shown in \cite{jl}, \cite{mp} (as well as in several subsequent papers in the literature)
for the problem $- \Delta u = \mu f(u)$ in  bounded domains of $\R{n}$. Differently from these results,
we focus on the Dirichlet energy rather than on the parameter $\mu$, and we deal with a faster growth of the nonlinearity.

\

In order to prove Theorem \ref{nonex} we first notice that a critical point of $E|_{M_\Lambda}$ solves
\begin{equation}\label{pb0}
  \left\{
    \begin{array}{ll}
      - \D u = \l u e^{u^2} & \hbox{in } \Omega; \\
      u = 0  & \text{on } \partial \Omega;\\
      \|u\|_{H^1_0}^2=\Lambda, &
    \end{array}
  \right.
\end{equation}
for some $\lambda>0$, and that when $\Omega=B_1$ a positive solution to \eqref{pb0} is radially symmetric by Theorem 1 in \cite{GNN}. Then it will be crucial to understand the blow-up behavior of a sequence of symmetric positive solutions to \eqref{pb0}, i.e. solutions $u_k$ to

\begin{equation}\label{eqk}
\left\{
\begin{array}{ll}
- \Delta u_k = \lambda_k u_k e^{u_k^2} &\text{in }  B_1; \\
u_k=0&\text{on } \de B_1;\\
u_k>0&\text{in } B_1;\\
\|u_k\|^2_{H^1_0}=\Lambda_k.&
\end{array}
\right.
\end{equation}

In this direction a lot of work has already been done. For the sake of simplicity we shall present only the radially symmetric versions of the results which we quote, referring to the original papers for the general cases. For instance O. Druet proved (see also \cite{AS} and
\cite{AD} for previous related results, where the blow-up profile was identified, and \cite{LRS} where the parabolic case was treated): let $(u_k)$ be a sequence of solutions to \eqref{eqk} with $\Lambda_k\le C$ and $\sup_{B_1} u_k\to\infty$. Then up to a subsequence $\lambda_k\to\lambda_\infty\in [0,2\pi]$,\footnote{The constant $2\pi$ is the first eigenvalue of $-\Delta$ on $B_1$. As proven by Adimurthi \cite{adi}, Problem \eqref{pb0} has a positive solution if and only if $\lambda\in (0,\lambda_1(\Omega))$, where $\lambda_1(\Omega)$ is the first eigenvalue of $-\Delta$ on $\Omega$ with Dirichlet boundary condition.}  $u_k\to u_\infty$ strongly in $C^1_{\loc}(\overline B_1\setminus \{0\})$ and weakly in $H^1_0(B_1)$, where
\begin{equation}\label{eq:inf}
-\Delta u_\infty=\lambda_\infty u_\infty e^{u_\infty^2}\quad \text{in }B_1,
\end{equation}
and $\Lambda_k\to 4\pi L+\|u_\infty\|^2_{H^1_0}$ for some integer $L\ge 1$. More precisely
\begin{equation}\label{quantdru}
|\nabla u_k|^2 dx\rightharpoonup 4\pi L \delta_0+|\nabla u_\infty|^2dx,\quad   \lambda_k u_k^2 e^{u_k^2}dx \rightharpoonup 4\pi L \delta_0+ \lambda_\infty u_\infty^2 e^{u_\infty^2}dx
\end{equation}
weakly in the sense of measures. The questions whether $u_\infty$ and $\lambda_\infty$ can actually be non-zero, and whether $L$ can be greater than one (i.e. whether the blow-up can be non-simple, using a terminology introduced in \cite{scho}) were left open  (in fact also higher dimensional generalization of the result of Druet, see e.g. \cite{mar2}, \cite{MS} and \cite{struwe}, produced analogous open questions), but we are now able to give a negative answer to both questions, as stated in the next theorem.

\begin{trm}\label{t:comp} Let $u_k$ be a sequence of solutions to \eqref{eqk}.
Then up to extracting a subsequence we have for $k\to\infty$ either
\begin{itemize}
\item[(i)] $\lambda_k\to\lambda_\infty\in [0,2\pi]$, $u_k\to u_\infty$ in $C^{1}(\overline B_1)$, where $u_\infty$ solves \eqref{eq:inf},
or
\item[(ii)] $\lambda_k\to 0$, $\|u_k\|^2_{H^1_0}\to 4\pi$, $u_k\to 0$ weakly in $H^1_0(B_1)$ and strongly in $C^1_{\loc}(\overline B_1 \setminus\{0\})$ and
\begin{equation}\label{quant0}
|\nabla u_k|^2 dx\rightharpoonup 4\pi \delta_0, \quad \lambda_k u_k^2e^{u_k^2} dx \rightharpoonup 4\pi \delta_0
\end{equation}
weakly in the sense of measures.
\end{itemize}
\end{trm}

The proof of Theorem \ref{t:comp} is self-contained. In some parts we could have used previous results of \cite{AS} or \cite{Dru}, but these hold for general domains, and consequently their proofs are more involved and we did not want to rest on them. Our main argument is not based on a Pohozaev-type identity as the results in \cite{Dru}, \cite{LRS}, \cite{MS}  and \cite{struwe}, but on a simpler decay estimate of $u_k$ away from the blow-up point, which has some partial analogies with Lemma 3 of \cite{LS} (originating in \cite{sha}, see also \cite{BLS}). Notice that, contrary to the previous works, e.g. \cite{Dru}, in our Theorem \ref{t:comp} we do \emph{not} assume uniform bounds on $\|u_k\|^2_{H^1_0}$, i.e. $\Lambda_k\le C$. This is crucial if we want to apply Theorem \ref{t:comp} to prove Theorem \ref{nonex}.

\medskip

The final picture that we get is then much closer to the geometric situation of the Liouville equation as studied by Brezis-Merle, Li-Shafrir and Li. More precisely, and working again on $B_1$ for simplicity, consider a sequence $(v_k)$ of radially symmetric solutions to
$$-\Delta v_k= V_ke^{2v_k} \text{ in } B_1\subset\R{2},\quad V_k\to V_0>0\text{ in }C^0(\overline B_1),\quad \|e^{2v_k}\|_{L^1}\le C, \quad \sup_{B_1} v_k\to +\infty.$$
Then, as proven in \cite[Theorem 3]{BM},
\begin{equation}\label{ukin}
v_k\to -\infty\text{ uniformly locally in }B_1\setminus \{0\}
\end{equation}
and
\begin{equation}\label{quantlio}
\quad V_ke^{2v_k}dx \rightharpoonup \alpha \delta_0 \quad \text{weakly as measures, for some }\alpha\ge 2\pi.
\end{equation}
Here $V_ke^{2v_k}$ plays the role of the energy density $\lambda_k u_k^2e^{u_k^2}$ from \eqref{quantdru} and \eqref{quant0}. Then \eqref{ukin} and \eqref{quantlio} are the equivalent of $\lambda_\infty u_\infty^2 e^{u_\infty^2}dx=0$ in \eqref{quant0} (compared with \eqref{quantdru}).

Y-Y. Li and I. Shafrir, see \cite{yanyan},\cite{LS}, complemented the result of Brezis-Merle by showing that $\alpha=4\pi$ in \eqref{quantlio}, finally yielding $V_ke^{2v_k}\rightharpoonup 4\pi\delta_0$, in analogy with \eqref{quant0}.
On the other hand we remark that the proof of \eqref{quant0} is more subtle because the nonlinearity $ue^{u^2}$ is more difficult to handle than $e^{2v}$. In fact, as already noticed in previous works, e.g. \cite{AS}, suitable scalings $\eta_k$ of blowing-up solutions of \eqref{eqk} converge in $C^1_{\loc}(\R{2})$ to a solution $\eta_0$ of $-\Delta v=4e^{2v}$, see Lemma \ref{etamu2}. Unfortunately this information is too weak for our purposes and we need to linearize the equation satisfied by $\eta_k$ (Eq. \eqref{eq:etak} below) to better understand its asymptotics (Lemma \ref{lemma:w}), and to have a global estimate of $\eta_k$ (Lemma \ref{lemma:stima}).

\medskip

We also point out that an immediate consequence of the proof of Theorem \ref{nonex} is the existence
of blowing-up solutions to \eqref{eqk} with bounded energies ($\L_k \to 4 \pi$). This has long been an open problem: Adimurthi and Prashanth \cite{AP} were only able to prove the existence of blowing-up Palais-Smale sequences, while more recently Del Pino, Musso and Ruf, with an  approach technically much richer, showed that blowing-up solutions exist for any domain $\Omega$, see \cite{DMR2}. Our method applies only to the unit disk, but it is on the other hand relatively elementary and explicit. 

\

\noindent In the following the letter $C$ denotes a large constant which may change from line to line and even within the same line.

\

\begin{center}

 {\bf Acknowledgements}

\end{center}

\noindent The authors are supported by the project FIRB-Ideas
{\em Analysis and Beyond}, and L.M. was partially
supported by the Swiss National Fond Grant no. PBEZP2-129520.
During the preparation of this work
L.M. was hosted by SISSA, while A.M. by Centro de Giorgi in
Pisa. They are grateful to these institutions for the kind hospitality.

\section{Proof of Theorem \ref{t:comp}}

By \cite[Theorem 1]{GNN} a positive solution $u$ to \eqref{pb0} is radially symmetric.   With a little abuse of notation we shall write $u(x)=u(r)$ for $x\in B_1$ with $|x|=r$. Since $\Delta u\le 0$
$$2\pi u'(r)=\int_{B_r}\Delta udx < 0; \qquad r > 0,$$
hence $u(r)$ is decreasing.

Consider now a sequence $u_k$ as in the statement of the theorem. By elliptic estimates, if $\max_{B_1}u_k\le C$, then we are in case (i). Let us therefore assume that, up to a subsequence
\begin{equation}\label{ublow}
\mu_k:=u_k(0)=\max_{B_1} u_k \to \infty\quad \text{as }k\to\infty.
\end{equation}

\begin{lemma}\label{etamu2} Let $r_k>0$ be such that $r_k^2 \lambda_k \mu_k^2 e^{\mu_k^2}=4$. Then as $k\to\infty$ we have $r_k\to 0$,
\begin{equation}\label{etamu}
\eta_k(x):=\mu_k (u_k(r_k x)-\mu_k) \to \eta_0(x):= -\log(1+|x|^2)\quad \text{in }C^1_{\loc}(\R{2}),
\end{equation}
and
\begin{equation}\label{enerbolla}
\lim_{R\to\infty}\lim_{k\to\infty}\int_{B_{Rr_k}}\lambda_k u_k^2 e^{u_k^2}dx=\int_{\R{2}}4e^{2\eta_0}dx=4\pi.
\end{equation}
\end{lemma}

\begin{proof} We first prove that $\lim_{k\to \infty} r_k=0$. Otherwise up to extracting a subsequence we have $\lambda_k \mu_k^2 e^{\mu_k^2}\le C$. Then, using that $u_k'\le 0$ in $[0,1]$ we see that
$$-\mu_k\Delta u_k \le \lambda_k \mu_k^2e^{\mu_k^2}\le C\quad \text{in }B_1.$$
Therefore as $k\to \infty$ we get $\Delta u_k\to 0$ uniformly and by elliptic estimates $u_k\to 0$ in $C^1(\overline B_1)$, contradicting \eqref{ublow}.

Set now $v_k(x)=u_k(r_k x)- \mu_k$. We claim that $v_k\to 0$ in $C^{1}_{\loc}(\R{2})$ as $k\to\infty$. Indeed
$$-\Delta v_k(x)= \frac{4}{\mu_k} \frac{u_k(r_k x)}{\mu_k}e^{(u(r_k x)-\mu_k)(u(r_k x)+\mu_k)}\to 0\quad \text{uniformly as }k\to \infty$$
since $4/\mu_k\to 0$, $0\le u_k/\mu_k\le 1$ and $(u_k(r_k x)-\mu_k)(u_k(r_k x)+\mu_k)\le 0$. Now notice that $v_k\le 0$ and $v_k(0)=0$. Then the Harnack inequality implies the claim.

Therefore we have
$$-\Delta \eta_k =V_k e^{2 a_k \eta_k}\quad\text{in }B_{1/r_k},$$
where
$$V_k(x)=\frac{4 u_k(r_k x)}{\mu_k}\to 4,\quad a_k=\frac{1}{2}\bigg(\frac{u_k(r_k x)}{\mu_k}+1\bigg)\to 1\quad \text{in }C^0_{\loc}(\R{2}).$$
Considering that $\eta_k\le 0$, $\Delta \eta_k$ is locally bounded and $\eta_k(0)=0$ we have $\eta_k\to \eta^*$ in $C^{1}_{\loc}(\R{2})$ by the Harnack inequality, where $-\Delta \eta^*=4e^{2\eta^*}$ and $\eta^*(0)=0$. On the other hand
\begin{equation}\label{eq:eta0}
-\Delta \eta_0=4e^{2\eta_0}\text{ in }\R{2},\quad \eta_0(0)=0,
\end{equation}
hence it follows from the uniqueness of solutions
to the Cauchy problem (recall that all functions here are radially symmetric) that $\eta^*=\eta_0$.

Finally \eqref{enerbolla} follows from Fatou's lemma.
\end{proof}

Notice that
\begin{equation}\label{eq:etak}
-\Delta \eta_k =4\bigg(1+\frac{\eta_k}{\mu_k^2}\bigg) e^{\Big(2+\frac{\eta_k}{\mu_k^2}\Big)\eta_k}.
\end{equation}

\begin{lemma}\label{lemma:w} Set $w_k:=\mu_k^2(\eta_k-\eta_0).$ Then we have $w_k\to w$ in $C^1_{\loc}(\R{2}),$
where
\begin{equation}\label{defw}
w(r) := \eta_0(r) + \frac{2r^2}{1+r^2}-\frac{1}{2}\eta_0^2(r)+\frac{1-r^2}{1+r^2}\int_1^{1+r^2} \frac{\log t}{1-t}dt
\end{equation}
is the unique solution to the ODE
\begin{equation}\label{eqw}
-\Delta w= 4e^{2\eta_0}(\eta_0+\eta_0^2+2w),\quad w(0)=0,\quad w'(0)=0.
\end{equation}
Moreover $w$ satisfies
\begin{equation}\label{4pi}
\int_{\R{2}}\Delta wdx=-4\pi,
\end{equation}
and
\begin{equation}\label{wbound}
\sup_{r\in [0,\infty)}|w(r) -\eta_0(r)|<\infty.
\end{equation}
\end{lemma}

\begin{proof} Set $\ve_k:=\mu_k^{-2}\to 0$ as $k\to \infty$.
Using \eqref{eq:eta0} and \eqref{eq:etak} we compute
\begin{equation}\label{eqDw}
\begin{split}
-\Delta w_k= & \frac{1}{\ve_k}\left[4(1+\ve_k\eta_k)e^{(2+\ve_k \eta_k)\eta_k}-4 e^{2\eta_0} \right]\\
=&\frac{4e^{2\eta_0}}{\ve_k}\left[   (1+\ve_k(\eta_0+(\eta_k-\eta_0)))e^{2(\eta_k-\eta_0)+\ve_k\eta_0^2+2\ve_k\eta_0(\eta_k-\eta_0)+\ve_k(\eta_k-\eta_0)^2}-1 \right].
\end{split}
\end{equation}
By Lemma \ref{etamu2} for every $R>0$ we have $\eta_k(r)-\eta_0(r)=o(1)\to 0$ as $k\to \infty$  uniformly for $r\in [0,R]$, and we can use a Taylor expansion:
$$e^{2(\eta_k-\eta_0)+\ve_k\eta_0^2+2\ve_k\eta_0(\eta_k-\eta_0)+\ve_k(\eta_k-\eta_0)^2}=
1+2(\eta_k-\eta_0)+\ve_k\eta_0^2 +o(1)\ve_k +o(1)(\eta_k-\eta_0),$$
with errors $o(1)\to 0$ as $k\to \infty$ uniformly for $r\in [0,R]$. Going back to \eqref{eqDw} we get
$$-\Delta w_k=4e^{2\eta_0}[\eta_0+\eta_0^2+2w_k +o(1)+o(1)w_k],$$
with $o(1)\to 0$ as $k\to \infty$  uniformly for $r\in [0,R]$.
By ODE theory $w_k(r)$ is locally bounded, and by elliptic estimates $w_k\to \tilde w$ in $C^1_{\loc}(\R{2})$, where $\tilde w$ satisfies \eqref{eqw}.

Since the solution to the Cauchy problem \eqref{eqw} is unique, in order to prove that $\tilde w=w$ (with $w$ given in \eqref{defw}) it is enough to show that $w$ solves \eqref{eqw}.
It is easily seen that $w(0)=0$. First computing
\begin{equation*}
\Big(-\frac{1}{2}\eta_0^2(r)\Big)'+\frac{1-r^2}{1+r^2}\frac{d}{dr}\int_1^{1+r^2} \frac{\log t}{1-t}dt=-\frac{2\log(1+r^2)}{r(1+r^2)},
\end{equation*}
we get
\begin{equation}\label{eq:w'}
w'(r)=\frac{2r(1-r^2)}{(1+r^2)^2}-\frac{2\log(1+r^2)}{r(1+r^2)}-\frac{4r}{(1+r^2)^2} \int_1^{1+r^2} \frac{\log t}{1-t}dt,
\end{equation}
$w'(0)=0$,
and using $\Delta w(r)=w''(r)+\frac{w'(r)}{r}$ we finally get
\begin{equation*}
\begin{split}
-\Delta w(r)&= \frac{16 r^2}{(1+r^2)^3}-\frac{12\log(1+r^2)}{(1+r^2)^2}+\frac{8(1-r^2)}{(1+r^2)^3}\int_1^{1+r^2} \frac{\log t}{1-t}dt\\
&=4e^{2\eta_0}\bigg[\frac{4r^2}{1+r^2}+3\eta_0 +2\frac{1-r^2}{1+r^2}\int_1^{1+r^2} \frac{\log t}{1-t}dt\bigg]=4e^{2\eta_0}[\eta_0+\eta_0^2+2w].
\end{split}
\end{equation*}
To prove \eqref{4pi} we use the divergence theorem and \eqref{eq:w'} to get
$$\int_{\R{2}}\Delta wdx=\lim_{r\to\infty}2\pi r w'(r)=-4\pi.$$
Similarly from \eqref{eq:w'} we bound
$$|w'(r)-\eta_0'(r)|\le \frac{C}{1+r^2}\quad \text{for }r\in [0,\infty),$$
and integrating in $r$ also \eqref{wbound} follows.
\end{proof}

Lemma \ref{etamu2} tells us that for $R>0$ and $k\ge k_0(R)$ we have $\eta_k\to \eta_0$ in $C^1(B_R)$. On the other hand $\eta_k$ is defined on $B_{r_k^{-1}}$ with $r_k^{-1}\to \infty$ as $k\to\infty$, and Lemma \ref{etamu2} gives us no information on the behavior of $\eta_k$ on $B_{r_k^{-1}}\setminus B_R$. We shall now use Lemma \ref{lemma:w} to fill this gap and have a crucial estimate of $\eta_k$ in all of $B_{r_k^{-1}}$.

\begin{lemma}\label{lemma:stima} Fix $R_0\in (0,\infty)$ such that $w\le -1$ on $[R_0,\infty)$, where $w$ is given by \eqref{defw} and such $R_0$ exists thanks to \eqref{wbound}. Then for $k$ large enough
\begin{equation}\label{eq:cruc3}
\eta_k(r) \le \eta_0(r),\quad \text{for } r\in [R_0, r_k^{-1}],
\end{equation}
or equivalently
\begin{equation}\label{eq:cruc}
u_k(r) \le \mu_k -\frac{1}{\mu_k} \log \bigg(1+\Big( \frac r {r_k}\Big)^2\bigg),\quad \text{for } r\in [R_0r_k,1].\end{equation}
\end{lemma}

\begin{proof}
Write $\ve_k:=\mu_k^{-2}=u_k^{-2}(0)$ and $ \eta_k = \eta_0 + \e_k w + \phi_k$.
Then \eqref{eq:etak} is equivalent to
\begin{equation*}
-\Delta  \eta_k = 4 \left( 1 + \e_k (\eta_0 +   \e_k w + \phi_k ) \right) e^{ (2 + \e_k (\eta_0 + \e_k w +
  \phi_k)) (\eta_0 + \e_k w + \phi_k )},
\end{equation*}
and taking \eqref{eq:eta0} and \eqref{eqw} into account we find
$$-\Delta \phi_k=\Phi_k(\phi_k),$$
where for any function $\phi$
$$
 \Phi_k(\phi) := 4 (1 + \e_k (\eta_0 + \e_k w + \phi)) e^{(2 + \e_k (\eta_0
  + \e_k w + \phi)) (\eta_0 + \e_k w + \phi)} - 4 e^{2 \eta_0} - 4 e^{2 \eta_0}
  \e_k \left[ \eta_0 + \eta_0^2  + 2 w \right].
$$
We now expand
\begin{equation*}
\begin{split}
  (2 + \e_k (\eta_0 + \e_k w +
  \phi)) (\eta_0 + \e_k w + \phi)  =& 2 \eta_0
  + 2 \e_k w +  2 \phi+\e_k \eta_0^2   + 2 \e_k^2 w \eta_0 + \e_k^3 w^2 \\
&+ \e_k \phi   \left(2 \eta_0 +2 \e_k w + \phi \right) =:2\eta_0+h_k(\phi).
\end{split}
\end{equation*}
To avoid cumbersome notations we will also write, for a given function $\phi$ and any given $k$,
\begin{equation*}
\begin{split}
h&:=h_k(\phi)=  2 \e_k w +  2 \phi+\e_k \eta_0^2   + 2 \e_k^2 w \eta_0 + \e_k^3 w^2 + \e_k \phi   \left(2 \eta_0 +2 \e_k w + \phi \right), \\
\eta&:= \eta_0+\ve_k w+\phi
\end{split}
\end{equation*}
so that
$$
   e^{(2 + \e_k \eta) \eta} = e^{2 \eta_0 + h},\quad \Phi_k(\phi)=4e^{2\eta_0}\left[(1+\ve_k\eta)e^h -1-\ve_k\eta_0 -\ve_k \eta_0^2-2\ve_k w \right].
$$
Then with a Taylor expansion we can write
$$
  e^{(2 + \e_k \eta) \eta}  = e^{2 \eta_0} \left[  1 + h + O(h^2) \right],
$$
where $|O(h^2) |\leq C h^2$ for a fixed positive constant $C$, provided $|h| \leq  1$. Then, using
\eqref{wbound} to bound $|w(r)|\le C(1+\log(1+r^2))$,
\begin{eqnarray*}
 (1+\ve_k \eta)e^{h} & = & 1 +
  \e_k \eta_0 + 2 \e_k w + \e_k \eta_0^2 + 2 \phi   \\
   & &+   O(\phi) \left( O(\phi) + O(\e_k (1+\log (1+r^2))^2)
   \right) + O(\e_k^2 (1+\log(1+r^2))^3)
\end{eqnarray*}
where $|O(s)| \leq C s$. Then
\begin{equation}\label{eq:phi}
    \Phi_k (\phi) = 4 e^{2 \eta_0} \left[ 2 \phi + O(\phi) \left(
    O(\phi) + O(\e_k (1+\log (1+r^2))^2)
   \right) + O(\e_k^2 (1+\log(1+r^2))^3)  \right],
\end{equation}
as long as $|h|\le 1$, which is true provided for some $\delta>0$ small enough
\begin{equation}\label{eq:expvalid}
    |\phi| \leq \delta; \quad \quad  \e_k (1+\log (1+r^2))^2 \leq \delta.
\end{equation}
Similarly if $\tilde{\phi}$ is another function with $|\tilde{\phi}(r)|\le \delta$ one has
\begin{equation}\label{eq:diffphi}
    |\Phi_k(\phi) - \Phi_k(\tilde{\phi})| \leq 4 e^{2 \eta_0} \left[ 2
  |\phi - \tilde{\phi}|  +O(\phi-\tilde \phi)\left(|\phi+\tilde \phi|+ O(\e_k (1+\log (1+r^2))^2)\right)   \right].
\end{equation}

We shall now use the contraction mapping theorem to bound $\phi_k$.
We restrict our attention to an interval $[0,s_k]$ with $s_k=o(1)e^{\mu_k}$ and to functions $\phi: [0,s_k]\to\R{}$ satisfying $\phi(r)\le O(\ve_k^2)(1 + \log(1 + r^2))$, so that \eqref{eq:phi} and \eqref{eq:expvalid} hold for $k$ large enough. With these restrictions \eqref{eq:diffphi} gives
\begin{equation}\label{eq:diffphi2}
    |\Phi_k(\phi) - \Phi_k(\tilde{\phi})| \leq (8 +o(1))e^{2 \eta_0}   |\phi - \tilde{\phi}|,
\end{equation}
with error $o(1)\to 0$ as $k\to\infty$.


\medskip

By the above computations, $\eta_k = \eta_0 + \e w + \phi_k$ solves
\eqref{eq:etak} if and only if $\phi_k$ satisfies
$$
    - \D \phi_k = \Phi_k(\phi_k); \qquad \quad \phi_k(0) = 0,\quad \phi_k'(0)=0.
$$
Setting $\phi = \phi_k$ and
$
  \psi = r \phi',
$
the last equation gives the system
\begin{equation}\label{eq:sys0}
    \left\{
     \begin{array}{ll}
       \phi' & = \frac{1}{r} \psi,   \\
       \psi' & =-r \Phi_k(\phi);
     \end{array} \right. \qquad \quad (\phi(0),\psi(0)) = (0,0).
\end{equation}
The solutions of \eqref{eq:sys0} are the fixed point of some integral equation.
For technical reasons, it will be convenient to integrate
starting from some value $T > 0$ (to be fixed later) of the $r$ parameter rather than from $r=0$.
If we let \eqref{eq:sys0} evolve up to time $T$, by the smooth dependence
on initial data then (for $\e_k$ small) the solution will satisfy
\begin{equation}\label{eq:phiT}
  |\phi(r)| \leq C(T) \e_k^2, \qquad \quad |\psi(r)| \leq C(T) \e_k^2, \qquad\quad \text{for }r\in [0,T],
\end{equation}
uniformly in $\e_k$. Notice that $\phi(T)=\phi_k(T)$ and $\phi(T)=T\phi_k'(T)$.

We consider then the functions
$$
  F_{1,(\phi,\psi)}(r) := \phi(T) + \int_T^r \psi(s)
  \frac{ds}{s}, \qquad \quad  r \geq T;
$$
$$
  F_{2,(\phi,\psi)}(r) := \psi(T) - \int_T^r s \Phi_k(\phi)(s) ds, \qquad \quad  r \geq T.
$$
Fixing $S=s_k > T$, with $s_k=o(1)e^{\mu_k}$, we next define the norms
$$
  \|f\|_1 = \sup_{r \in (T,S]} \left| \frac{f(r)}{\log r
  - \log T} \right|; \qquad \quad \|f\|_2 = 2\sup_{r\in [T,S]} |f(r)|.
$$
For a large constant $\tilde{C} > 0$ to be fixed later, we will work with the following
set of functions
$$
  \mathcal{B}_{\tilde{C}} = \left\{ (\phi,\psi) \; : \; \|\phi - \phi_k(T)\|_1
  \leq \tilde{C} \e_k^2,\; \|\psi\|_2 \leq \tilde{C} \e_k^2,\; \phi(T) = 
  \phi_k(T), \psi(T)=T\phi_k'(T) \right\}.
$$

\

\noindent We now check next that the map $(\phi,\psi)\mapsto (F_{1,(\phi,\psi)}, F_{2,(\phi,\psi)})$ sends
$\mathcal{B}_{\tilde{C}}$ in itself,
for suitable choices of $\tilde{C}$ and $T$, and that it is a contraction.
In fact, for $(\phi,\psi) \in \mathcal{B}_{\tilde{C}}$ one has that
$$
  |F_{1,(\phi,\psi)}(r) - \phi(T)| \leq \frac 12 \tilde{C} \e_k^2 (\log r - \log T),
$$
which implies $\|F_{1,(\phi,\psi)} - \phi(T)\|_1 \leq \frac 12  \tilde{C} \e_k^2$,
as desired.

Moreover by \eqref{eq:phi} and \eqref{eq:phiT} one has that
\begin{equation*}
\begin{split}
  |F_{2,(\phi,\psi)}(r)| \le&  |\psi(T)|+\int_T^r s4e^{2\eta_0(s)}\big[2\phi(s)+O(\phi(s))(O(\phi(s)) +O(\ve_k(1+\log(1+s^2))^2)    \\
& +O(\ve_k^2(1+\log(1+s^2))^3)\big]ds \\
\leq& C(T) \e_k^2 + 8 \int_T^\infty \frac{s (\phi(s)(1+o(1)) }{(1+s^2)^2} ds +\int_T^\infty \frac{C_0 \ve_k^2s(1+\log(1+s^2))^3)}{(1+s^2)^2} ds\\
\leq& C(T) \e_k^2 + 9 \int_T^\infty \frac{s \ve_k^2(C(T)+\tilde C \log s )}{(1+s^2)^2} ds +C_0\ve_k^2 \int_T^\infty \frac{s (1+\log(1+s^2))^3}{(1+s^2)^2} ds\\
\le & \ve_k^2\bigg[ C(T)\bigg(1+9\int_T^\infty\frac{s}{(1+s^2)^2} ds\bigg) \\
&+9\tilde C \int_T^\infty\frac{s\log s }{(1+s^2)^2}ds +C_0\int_T^\infty \frac{s(1+\log(1+s^2))^3}{(1+s^2)^2}ds\bigg]
\end{split}
\end{equation*}
for some fixed $C_0$ independent of $\tilde{C}$ and $\e_k$. Now first choosing $T\ge 1$ so large that
\begin{equation}\label{scegliT}
9\int_T^\infty\frac{s\log s ds}{(1+s^2)^2}< \frac{1}{2},
\end{equation}
and then $\tilde C$ large enough compared to $C(T)$ and $C_0$, we obtain
$$
    \|F_{2,(\phi,\psi)}\|_2 \leq \tilde{C} \e_k^2,
$$
so we are done showing that $(F_{1,(\cdot,\cdot)}, F_{2,(\cdot,\cdot)})$ maps $\mathcal{B}_{\tilde{C}}$ in itself.

Let us verify that $F$ is a contraction. We easily estimate
for $(\phi, \psi), (\tilde{\phi}, \tilde{\psi}) \in \mathcal{B}_{\tilde{C}}$
$$
   \|F_{1,(\phi,\psi)} - F_{1,(\tilde{\phi},\tilde{\psi})}\|_1 \leq \frac{1}{2}
   \| \psi - \tilde{\psi} \|_2.
$$
Using \eqref{eq:diffphi2} and \eqref{scegliT} we also find for $k$ large enough
\begin{equation*}
\begin{split}
  \|F_{2,(\phi,\psi)} - F_{2,(\tilde{\phi},\tilde{\psi})} \|_2 &\le 9\int_T^S\frac{s|\phi(s)-\tilde \phi(s)|}{(1+s^2)^2}ds \\
&\le 9\|\phi-\tilde\phi\|_1\int_T^S\frac{s(\log s-\log T)}{(1+s^2)^2}ds \leq \frac{1}{2}
   \|\phi - \tilde{\phi}\|_1,
\end{split}
\end{equation*}
so we have that indeed $F$ is a contraction.
In particular the map $(\phi,\psi)\mapsto (F_{1,(\phi,\psi)}, F_{2,(\phi,\psi)})$ has a fixed point in $(\overline\phi,\overline\psi)\in  \mathcal{B}_{\tilde C}$, which satisfies \eqref{eq:sys0}. Then, by uniqueness for the Cauchy problem, we have $(\overline\phi(r),\overline\psi(r))=(\phi_k(r),r\phi'_k(r))$ for $r\in [T,S]$, whence the bounds
\begin{equation}\label{stimaphi}
\phi_k(r)\le C(T)\ve_k^2+\tilde C\ve_k^2(\log r-\log T),\quad \phi_k'(r)\le \frac{\tilde C\ve_k^2}{2r},\quad \text{for } T\le r\le S=o(1)e^{\mu_k}.
\end{equation}

For every $k$ large enough fix now $S=s_k=o(1)e^{\mu_k}$ such that $s_k\ge 2\mu_k$. From \eqref{eq:phiT}, \eqref{stimaphi} and our choice of $R_0$, we get for $k$ large enough
\begin{equation}\label{eq:R0sk}
\eta_k(r)\le \eta_0(r)-\ve_k+(C(T)+\tilde C\log r)\ve_k^2<\eta_0(r),\quad \text{for } r\in [R_0,s_k].
\end{equation}
We shall now prove that
$$\int_{B_{s_k}} \Delta \eta_k dx <-4\pi$$
for $k$ large enough.
Indeed we have
$$-\int_{B_{s_k}}\Delta \eta_0dx=\int_{B_{s_k}}\frac{4}{(1+r^2)^2}dx=4\pi \bigg(1-\frac{1}{1+s_k^2}\bigg)=4\pi -\frac{4\pi}{s_k^2}+\frac{o(1)}{s_k^2}.$$
From \eqref{4pi} we have
$$-\int_{B_{s_k}} \ve_k\Delta wdx= 4\pi (1+o(1))\ve_k =\frac{4\pi (1+o(1))}{\mu_k^2}.$$
Finally, using \eqref{stimaphi} and the divergence theorem,
$$\bigg|\int_{B_{s_k}} \Delta \phi_k dx\bigg|= 2\pi s_k |\phi_k'(s_k)|=O(\ve_k^2)=(\mu_k^{-4}).$$
Summing up we infer
$$-\int_{B_{s_k}}\Delta \eta_k dx= 4\pi -\frac{4\pi(1+o(1))}{s_k^2}+\frac{4\pi (1+o(1))}{\mu_k^2}+O(\mu_k^{-4})>4\pi$$
for $k$ large enough, by our choice of $s_k$.
Since $\Delta \eta_k<0$ on all of $B_{1/r_k}$, we infer that
$$2\pi r \eta_k'(r)=\int_{B_r}\Delta\eta_k dx<\int_{B_{s_k}}\Delta\eta_k dx<-4\pi <\int_{B_r}\Delta \eta_0 dx=2\pi r \eta_0'(r),\quad r\in [s_k,1/r_k].$$
This, together with \eqref{eq:R0sk}, completes the proof of \eqref{eq:cruc3}.
\end{proof}

\medskip

\noindent\emph{Proof of Theorem \ref{t:comp} (completed).}
From \eqref{eq:cruc} it follows that $\lambda_k\to 0$ as $k\to\infty$. Indeed the function $-\frac{2}{\mu_k}\log \Big(\frac{r}{r_k}\Big)+\mu_k$ vanishes for $r=\rho_k=\frac{2}{\sqrt{\lambda_k}\mu_k}$.
Since $u_k>0$ in $B_1$ we must have for $k$ large enough $\rho_k\ge 1$, hence
\begin{equation}\label{lambdak}
\lambda_k\le \frac{4}{\mu_k^2} \to 0\quad \text{as } k\to\infty.
\end{equation}

Set $f_k:=\lambda_k u_k^2 e^{u_k^2}$. We now want to prove
\begin{equation}\label{eqsimple}
\lim_{R\to\infty}\lim_{k\to\infty}\int_{B_1\setminus B_{Rr_k}}f_kdx=0,
\end{equation}
which together with \eqref{enerbolla}, yields the convergence of $f_kdx$ in \eqref{quant0}.

Using the definition of $r_k$ and \eqref{eq:cruc} we have for $r\in [R_0r_k,1]$ and $k$ large enough
\begin{equation*}\begin{split}
r^2 f_k(r)&=4\bigg(\frac{r}{r_k}\bigg)^2\bigg(\frac{u_k(r)}{\mu_k}\bigg)^2e^{(1+u_k(r)/\mu_k)\eta_k(r/r_k)}\\
&\le 4\bigg(\frac{r}{r_k}\bigg)^2\bigg(\frac{u_k(r)}{\mu_k}\bigg)^2 \bigg( 1+ \bigg(\frac{r}{r_k}\bigg)^2\bigg)^{-1-\frac{u_k(r)}{\mu_k}} \\
&\le 4\bigg(\frac{u_k(r)}{\mu_k}\bigg)^2 \bigg( 1+ \bigg(\frac{r}{r_k}\bigg)^2\bigg)^{-\frac{u_k(r)}{\mu_k}}.
\end{split}
\end{equation*}
In order to further estimate the right-hand side, we notice that the function
$$\alpha_k(t):= t^2 \bigg( 1+ \bigg(\frac{r}{r_k}\bigg)^2\bigg)^{-t}$$
satisfies for any fixed $r>0$
$$\alpha_k\ge 0,\quad  \alpha_k(0)=0, \quad \lim_{t\to \infty }\alpha_k(t)=0\quad \alpha_k'(t)=0 \Leftrightarrow t=t_k:=\frac{2}{\log\Big(1+\Big(\frac{r}{r_k}\Big)^2\Big)}.$$
Hence $\alpha_k \le \alpha_k(t_k)$ and we conclude
$$r^2f_k(r)\le \frac{16}{\log^2\big(1+\big(\frac{r}{r_k}\big)^2\big)} \bigg(1+\Big(\frac{r}{r_k}\Big)^2\bigg)^{-2\log^{-1} (1+(r/r_k)^2)},\quad\text{for } r\ge R_0 r_k,\; k\text{ large},$$
which can be weakened to
\begin{equation}\label{eq:cruc5}
\log^2\Big(\frac{r}{r_k}\Big)^2  r^2f_k(r)\le \log^2\Big(1+\Big(\frac{r}{r_k}\Big)^2\Big) r^2f_k(r) \le 16, \quad\text{for } r\ge R_0r_k,\; k\text{ large}.
\end{equation}
Finally, it follows from \eqref{eq:cruc5} that
\begin{equation}\label{vanish}
\begin{split}
\lim_{R\to\infty}\lim_{k\to \infty}\int_{B_1\backslash B_{Rr_k}} f_k dx &\le \lim_{R\to\infty}\lim_{k\to \infty} \int_{Rr_k}^1 2\pi r f_k(r)dr   \le \lim_{R\to\infty}\lim_{k\to \infty}  \int_{Rr_k}^1 \frac{32\pi }{r[\log(r/r_k)]^2}dr\\
&=\lim_{R\to\infty}\lim_{k\to \infty} \bigg[-\frac{32\pi}{\log (r/r_k)}\bigg]_{Rr_k}^1=0.
\end{split}
\end{equation}
Then \eqref{eqsimple} is proven, and as already noticed  the first part of \eqref{quant0} follows.

Integrating by parts we also obtain
$$\|u_k\|^2_{H^1_0}=\int_{B_1}u_k(-\Delta u_k)dx=\int_{B_1}f_kdx\to 4\pi \quad\text{as }k\to\infty.$$
Then, up to extracting a subsequence we have $u_k\rightharpoonup u_\infty$ weakly in $H^1_0(B_1)$ for some function $u_\infty\in H^1_0(B_1)$.
Moreover, using that $\lambda_k\to 0$ as $k\to\infty$, we get for any $L>0$
$$\int_{B_1}\lambda_ku_ke^{u_k^2}dx\le \lambda_k\int_{\{x\in B_1:u_k(x)\le L\}}u_ke^{u_k^2}dx+\frac{1}{L}\int_{\{x\in B_1:u_k(x)< L\}}f_k dx\le o(1)C(L)+\frac{4\pi+o(1)}{L},$$
with error $o(1)\to 0$ as $k\to \infty$. Then, by letting $L\to\infty$ we infer that $\lambda_k u_ke^{u_k^2}\to 0$ in $L^1(B_1)$, and it follows that for every $\varphi\in C^1_c(B_1)$
$$\int_{B_1}(-\Delta u_\infty)\varphi dx= \lim_{k\to \infty}\int_{B_1}(-\Delta u_k)\varphi dx=\lim_{k\to \infty}\int_{B_1}-\lambda_k u_k e^{u_k^2}\varphi dx=0,$$
hence $-\Delta u_\infty=0$ in $B_1$, i.e. $u_\infty\equiv 0$.

This also implies the convergence of $|\nabla u_k|^2dx$ in \eqref{quant0}. Indeed, integrating by parts
and using that $u_k\to 0$ in $L^2(B_1)$ by the compactness of the embedding $H^1_0(B_1)\hookrightarrow L^2(B_1)$, we have for any $\varphi\in C^2_c( B_1)$
\begin{equation*}
\int_{B_1} |\nabla u_k|^2\varphi dx=\int_{B_1} \frac{\Delta(u_k^2)}{2}\varphi dx +\int_{B_1} f_k \varphi dx =\int_{B_1} \frac{u_k^2}{2} \Delta\varphi dx +\int_{B_1} f_k \varphi dx= \int_{B_1} f_k \varphi dx+o(1),
\end{equation*}
with error $o(1)\to 0$ as $k\to \infty$. Hence $f_kdx$ and $|\nabla u_k|^2 dx$ have the same weak limit in the sense of measures.

Now, using $u_k(1)=0$ and \eqref{quant0} 
we infer that $u_k\to 0$ in $L^\infty_{\loc}( B_1\setminus \{0\})$. Indeed for a fixed $\delta\in (0,1)$ and any $r\in [\delta,1]$ we have by H\"older's inequality
$$u_k(r)=u_k(r)-u_k(1)\le  \int_\delta^1|u_k'(\rho)|d\rho\le \|\nabla u_k\|_{L^2(B_1\setminus B_\delta)}\Big(\frac{1}{2\pi}\log\frac{1}{\delta}\Big)^{\frac{1}{2}}\to 0\quad \text{as } k \to \infty.$$
Then by elliptic estimates $u_k\to 0$ in $C^1_{\loc}( \overline B_1\setminus \{0\})$.
This completes the proof of Theorem \ref{t:comp}. \hfill $\square$

\section{Proof of Theorem \ref{nonex}}

Given $\mu,\lambda>0$ let $\overline u_{\mu,\lambda}\in C^\infty([0,T_{\mu,\lambda}))$ be the solution to the ODE
\begin{equation}\label{ODE}
-\frac{\de^2 \overline u}{\de r^2}-\frac{1}{r}\frac{\de \overline u}{\de r}=\lambda \overline u e^{\overline u^2},\quad \overline u(0)=\mu,\quad \overline u'(0)=0,
\end{equation}
where $[0,T_{\mu,\lambda})$, is the maximal interval of existence
for \eqref{ODE} (in fact $T_{\mu,\lambda}=\infty$, but we will not prove this). Then $u_{\mu,\lambda}(x):=\overline u_{\mu,\lambda}(|x|)$ satisfies
$$-\Delta u_{\mu,\lambda}=\lambda u_{\mu,\lambda}e^{u_{\mu,\lambda}^2}\quad \text{ in }B_{T_{\mu,\lambda}},
\qquad \quad u_{\mu,\lambda}(0)=\mu.$$
Set
$$\tau(\mu):=\inf\{r\in (0,T_{\mu,1}]:\overline u_{\mu,1}(r)=0\}.$$
We claim that $\tau(\mu)<\infty$ for every $\mu>0$. To see this, fix $r_0\in (0, \tau(\mu))$. Then for $r\in [ r_0,\tau(\mu)]$ the divergence theorem yields
$$\overline u_{\mu,1}'(r)=\frac{1}{2\pi r}\int_{B_r}\Delta u_{\mu,1}dx\le \frac{1}{2\pi r}\int_{B_{r_0}}\Delta u_{\mu,1}dx<-\frac{\ve}{r},$$
for some positive $\ve$. The claim easily follows from standard comparison arguments.

Now notice that, for any $\l, \l' > 0$
$$\overline u_{\mu,\lambda}\bigg(\sqrt{\frac{\lambda'}{\lambda}}r\bigg)= \overline u_{\mu,\lambda'}(r),$$
and for every $\mu>0$ set $\overline u_\mu:=\overline u_{\mu,\tau^2(\mu)}=\overline u_{\mu,\lambda_\mu}$, where $\lambda_\mu:=\tau^2(\mu)$. Then $\overline u_\mu$ is  positive in $[0,1)$, it solves \eqref{ODE} with $\lambda=\lambda_\mu$, and $\overline u_\mu(1)=0$.
By ODE theory the function $\Phi:\mu\mapsto \overline u_\mu|_{[0,1]}$ belongs to $C^0((0,\infty),C^2([0,1]))$.

Now given $\Lambda>0$ every non-negative critical point of $E|_{M_\Lambda}$ is smooth and satisfies
\begin{equation}\label{eqMT}
-\Delta u=\lambda ue^{u^2} \text{ in }B_1,\quad u=0\text{ on }\de B_1,
\end{equation}
for some $\lambda>0$.
By \cite[Theorem 1]{GNN} we have that $u$ is radially symmetric, i.e. we can write $u(x)=\overline u(|x|)$, where $\overline u$ satisfies \eqref{ODE} with $\mu=u(0)$ and the additional condition that $\overline u>0$ on $[0,1)$ and $\overline u(1)=0$. This is possible only if $\overline u=\overline u_{\mu}$.
Hence we have proven that every  solution $0\le u\not\equiv 0$ of \eqref{eqMT} is of the form $u(x)=u_\mu(x):=\overline u_\mu(|x|)$ for some $\mu> 0$. Define
$$\M{E}(\mu):= \|u_\mu\|_{H^1_0(B_1)}^2,\quad \Lambda^{\sharp}:=\sup_{\mu\in (0,\infty)}\M{E}(\mu).$$
We claim that $\Lambda^{\sharp}<\infty$. Indeed $\M{E}$ is continuous and it is clear that $u_\mu\to 0$ smoothly as $\mu\downarrow 0$, hence $\lim_{\mu\downarrow 0}\M{E}(\mu)=0$. Moreover Theorem \ref{t:comp} gives $\lim_{\mu\to\infty}\M{E}(\mu)=4\pi$, hence by continuity $\Lambda^{\sharp}<\infty$. This completes the proof of Part (i) of the theorem. By Theorem 1.7 in \cite{sIHP} it follows that $\Lambda^\sharp>4\pi$, and Parts (ii) and (iii) follow at once from the continuity of $\mathcal{E}$.
\hfill $\square$

\end{document}